\documentclass[11pt]{article}

\usepackage[T1]{fontenc}
\usepackage[utf8]{inputenc}
\usepackage{amsmath,amsthm,mathtools,bm}
\usepackage{newtxtext,newtxmath}
\usepackage{microtype}
\usepackage[a4paper,margin=27mm,headsep=8mm]{geometry}
\usepackage{booktabs,tabularx,array}
\usepackage{enumitem}
\usepackage{xcolor}
\usepackage{graphicx}
\usepackage{float}
\usepackage{tikz}
\usetikzlibrary{arrows.meta,positioning,calc}
\usepackage{tcolorbox}
\usepackage{natbib}
\usepackage{xurl}
\usepackage{hyperref}
\usepackage[nameinlink,capitalise,noabbrev]{cleveref}

\definecolor{deepblue}{RGB}{28,62,94}
\definecolor{softblue}{RGB}{235,241,246}
\definecolor{softgray}{RGB}{246,247,248}
\definecolor{mutedred}{RGB}{135,52,52}

\hypersetup{
  colorlinks=true,
  linkcolor=deepblue,
  citecolor=deepblue,
  urlcolor=deepblue,
  pdftitle={Prime-Exponent Transition Geometry and Divisor Barriers Between Consecutive Highly Composite Numbers},
  pdfauthor={Marco Mantovanelli},
  pdfsubject={Highly composite numbers and prime-exponent transition geometry},
  pdfkeywords={highly composite numbers, divisor function, prime-exponent lattice, bottleneck paths, discrete optimization}
}

\allowdisplaybreaks[2]
\setlist{itemsep=0.2em,topsep=0.3em}
\numberwithin{equation}{section}

\newtheoremstyle{blueplain}
  {0.7em}{0.7em}{\itshape}{}
  {\color{deepblue}\bfseries}{.}{0.55em}{}
\newtheoremstyle{bluedef}
  {0.7em}{0.7em}{\normalfont}{}
  {\color{deepblue}\bfseries}{.}{0.55em}{}
\theoremstyle{blueplain}
\newtheorem{theorem}{Theorem}[section]
\newtheorem{proposition}[theorem]{Proposition}
\newtheorem{lemma}[theorem]{Lemma}
\newtheorem{corollary}[theorem]{Corollary}
\newtheorem{conjecture}[theorem]{Conjecture}
\theoremstyle{bluedef}
\newtheorem{definition}[theorem]{Definition}
\newtheorem{remark}[theorem]{Remark}
\newtheorem{computation}[theorem]{Exact computation}

\crefname{theorem}{Theorem}{Theorems}
\crefname{proposition}{Proposition}{Propositions}
\crefname{lemma}{Lemma}{Lemmas}
\crefname{corollary}{Corollary}{Corollaries}
\crefname{conjecture}{Conjecture}{Conjectures}
\crefname{definition}{Definition}{Definitions}
\crefname{remark}{Remark}{Remarks}
\crefname{computation}{Exact computation}{Exact computations}

\newcommand{\N}{\mathbb N}

\newcommand{\Bcal}{\mathcal B}

\newcommand{\Gcal}{\mathcal G}

\newcommand{\vpp}[2]{v_{#1}(#2)}
\newcommand{\betageo}{\beta^{\mathrm{geo}}}
\newcommand{\supp}{\operatorname{supp}}

\newcommand{\zenododoi}{\href{https://doi.org/10.5281/zenodo.21983543}{\nolinkurl{10.5281/zenodo.21983543}}}

\begin{document}

\begin{center}
  {\fontsize{16.5}{20}\selectfont\bfseries
  \mbox{Prime-Exponent Transition Geometry and Divisor Barriers}\\[0.25em]
  \mbox{Between Consecutive Highly Composite Numbers}\par}
  \vspace{0.9em}
  {\large Marco Mantovanelli\par}
  \vspace{0.25em}
  {\normalsize\href{mailto:marco@mantovanelli.de}{\textcolor{black}{\nolinkurl{marco@mantovanelli.de}}}\par}
  \vspace{0.12em}
  {\normalsize\href{https://orcid.org/0009-0002-0631-293X}{\textcolor{black}{ORCID: 0009-0002-0631-293X}}\par}
\end{center}
\vspace{0.4em}

\begin{abstract}
Let $d(n)$ be the divisor function and let $H<H'$ be consecutive highly composite numbers.  We study directed unit moves between their prime-exponent vectors under the hard ceiling $z\le H'$.  The normalised capacity of such a geodesic is its smallest divisor count divided by $d(H)$, giving a finite fixed-endpoint maximin problem.

The exact record enumeration first finds a failure of the static surrogate $d(\gcd(H,H'))\ge d(H)/2$ at $48\,886\,437\,600<64\,250\,746\,560$, where the ratio is $4/9$.  For every state $z$ in the exponent box, however, we prove
\[
 d(z)d(HH'/z)\ge d(H)d(H')
\]
and deduce the record-box gap: no box state lies numerically strictly between the two records.  We also give an exact dynamic-programming recursion, solve the strata $L_-\le1$, and reduce the complete $L_-=2$ problem to an explicit divisor-selection functional.

 A computer-assisted enumeration through $10^{70}$ produces $889$ records and $888$ transitions.  Although the static half-gcd bound fails $119$ times, every computed geodesic capacity is at least $1/2$; equality occurs in exactly the $124$ transitions that lose an exponent-one support prime.  Independently checked certificates cover all $301$ transitions with $L_-=2$.  The corresponding universal half-capacity bound and equality classification remain open beyond the proved strata and the verified range.
\end{abstract}



\section{Introduction}
\label{sec:introduction}

Highly composite numbers are the strict record holders of the divisor function:
\begin{equation}
 H\ \text{is highly composite}
 \quad\Longleftrightarrow\quad
 d(m)<d(H)\quad(1\le m<H).
 \label{eq:hcn-definition-intro}
\end{equation}
Ramanujan introduced and analysed them as a concrete route into the maximal order of $d(n)$, together with the superior highly composite envelope that exposes the relevant prime-exponent thresholds \citep{Ramanujan1915,Ramanujan1997}.  The classical theory describes the shape of an individual record: its prime support is an initial segment, its exponents are non-increasing, and the exponent at a fixed prime is tightly controlled by the largest prime factor.  These themes were developed further by Alaoglu and Erd\H{o}s, Erd\H{o}s, Nicolas, Robin, and others \citep{AlaogluErdos1944,Erdos1944,Nicolas1971,Robin1983,NicolasRobin1983}.  Modern accounts and extensions continue to connect highly composite states with maximal-order problems, the Riemann zeta-function, and function-field analogues \citep{AndrewsBerndt2012,Tenenbaum2015,Afshar2021,Nicolas2022,Akatsuka2024,Teo2026}.

\subsection*{Relation to previous work}

\paragraph{Classical ingredients.}
Prime-exponent profiles, superior highly composite numbers, and constrained optimisation are classical tools in the subject.  Nicolas studied the distribution of ordinary highly composite numbers between consecutive superior highly composite numbers, using a benefit relative to a superior record to control exponent changes and to bound the record-counting function $Q(X)$ \citep{Nicolas1971}.  Robin subsequently formulated record generation as an integer-optimisation problem: his dynamic programme lists all highly composite numbers below a prescribed bound, while a second method combines Lagrange multipliers with Nicolas's benefit function to construct a block of records around a prescribed scale \citep{Robin1983}.  Nicolas and Robin used the superior highly composite envelope to determine the global maximum of $(\log d(n)\log\log n)/(\log 2\log n)$; geometrically, superior records occur as vertices of the concave envelope of the points $(\log n,\log d(n))$ \citep{NicolasRobin1983}.

\paragraph{Present contribution and scope of novelty.}
Dynamic programming itself is therefore not new in the theory of highly composite numbers.  Robin's programme optimises the endpoint integer under a size constraint, whereas the recursion used here fixes two consecutive ordinary records $H<H'$, restricts the intermediate states to their directed exponent box under the ceiling $z\le H'$, and maximises the smallest divisor count encountered along a geodesic.  We are not aware of an earlier formulation of this fixed-endpoint maximin problem in the cited literature.  The present paper develops the transition invariant $\betageo$, applies the complement map $z\mapsto HH'/z$ to obtain the record-box gap, and derives the exact two-deletion reduction and its finite certificates.  This limited comparison is not an absolute priority claim.  In particular, the coordinatewise divisor-product inequality is an elementary instance of discrete concavity; what is specific to the present argument is its application to consecutive divisor records and the resulting transition geometry.

Rather than describing a record $H$ in isolation, we now study the geometry between two consecutive records
\[
 H<H'.
\]
Write
\begin{equation}
 H=\prod_p p^{a_p},
 \qquad
 H'=\prod_p p^{b_p}.
 \label{eq:endpoint-vectors-intro}
\end{equation}
The vector $a=(a_p)_p$ must be rearranged into $b=(b_p)_p$.  We allow one elementary layer move at a time: an exponent is increased or decreased by one.  A path is constrained by the arithmetic ceiling $z\le H'$ at every intermediate state.  Along such a path, the multiplicative size $z$ plays the role of a hard resource constraint, while
\begin{equation}
 S(z):=\log d(z)=\sum_p\log(\vpp pz+1)
 \label{eq:divisor-entropy-intro}
\end{equation}
provides an additive quantity that we call \emph{divisor entropy}.  This is only a terminological analogy for $\log d(z)$; no Shannon entropy or underlying probability distribution is used.  The central maximin quantity asks how much of the old record entropy can be retained during the transition.

This setup is deliberately elementary.  It uses only unique factorisation, the divisor formula, and the strict record property.  Nevertheless, it produces a nontrivial dynamical structure.  The coordinatewise gcd is the lowest corner reached by the naive ``delete first, insert later'' route, but it need not be the best route.  Entropy-restoring insertions may be interleaved between compulsory deletions.  The resulting barrier problem is not a restatement of the static arithmetic of $\gcd(H,H')$.

\subsection*{The first surprise: the static half bound is false}

The most tempting conjecture is
\begin{equation}
 d(\gcd(H,H'))\ge\frac12d(H).
 \label{eq:half-gcd-intro}
\end{equation}
It fails.  Exact record enumeration shows that the first counterexample is
\begin{align}
 H&=48\,886\,437\,600
   =2^5 3^3 5^2 7^2 11\,13\,17\,19,\nonumber\\
 H'&=64\,250\,746\,560
   =2^6 3^3 5\,7\,11\,13\,17\,19\,23,
 \label{eq:first-counterexample-intro}
\end{align}
with $d(H)=3456$ and $d(\gcd(H,H'))=1536$.  Exact enumeration through $10^{70}$ finds $119$ such failures and a minimum observed static ratio $128/405$.

The transition itself does not fall below one half.  One optimal geodesic for the first counterexample is
\begin{equation}
\begin{split}
48\,886\,437\,600
&\longrightarrow 6\,983\,776\,800
\longrightarrow 13\,967\,553\,600\\
&\longrightarrow 2\,793\,510\,720
\longrightarrow 64\,250\,746\,560,
\end{split}
\label{eq:first-counterexample-path-intro}
\end{equation}
whose divisor counts are
\[
3456\longrightarrow2304\longrightarrow2688\longrightarrow1792\longrightarrow3584.
\]
The capacity is therefore $1792/3456=14/27>1/2$.  The insertion of a factor $2$ before the second deletion avoids the deeper gcd valley.  This distinction between a static overlap and an optimally scheduled path motivates the whole paper.

\subsection*{Main results and status}

The paper is organised around three logically different layers.  The distinction matters, because the finite experiment is extensive but does not replace a universal proof.

\begin{tcolorbox}[colback=softgray,colframe=deepblue!70,boxrule=0.7pt,arc=1.2mm,left=1.2mm,right=1.2mm,top=1.0mm,bottom=1.0mm]
\small
\textbf{Unconditional theorems.}
The exponent-box complement $z\mapsto HH'/z$ obeys a divisor-product inequality, which yields the record-box gap and the tunnelling corollary.  The geodesic capacity satisfies an exact acyclic dynamic-programming recursion.  Transitions with $L_-\le1$ are solved exactly.  For $L_-=2$ the full path problem reduces to a finite divisor-selection problem, and crossing, source-profile packing, and compensated matching give explicit half-capacity certificates.

\medskip
\textbf{Exact finite verification.}
A profile enumeration through $10^{70}$ yields $889$ records.  Exact integer arithmetic verifies the maximin capacity for all $888$ consecutive transitions.  The greedy scheduler is checked independently against those optima, with every logarithmic slope comparison certified by arbitrary-precision interval arithmetic.  Further computer-assisted certificate audits are identified separately below.

\medskip
\textbf{Open statements.}
The universal half-capacity inequality $\betageo(H,H')\ge1/2$ and the equivalence between equality and loss of an exponent-one support prime remain conjectural outside the proved strata and the verified range.
\end{tcolorbox}

The decisive structural theorem is a reflection principle.  If $z$ lies coordinatewise between $H$ and $H'$, then $z^\sharp=HH'/z$ lies in the same box and
\begin{equation}
 d(z)d(z^\sharp)\ge d(H)d(H').
 \label{eq:reflection-preview}
\end{equation}
If a box state were numerically between two consecutive records, then its complement would be there as well, and both divisor counts would be at most $d(H)$; this contradicts \eqref{eq:reflection-preview}.  Consequently every interior state of a ceiling-admissible mixed geodesic lies below $H$.  The path does not drift through the interval $(H,H')$; it leaves the old record downward and returns to the new record only at the final insertion.  The language of a barrier or a tunnel is therefore not decorative: it is an exact consequence of the record property.

For two deletions the geometry collapses further.  Set
\begin{equation}
 g=\gcd(H,H'),\qquad C=H'/g,\qquad
 \delta=\frac{d(g)}{d(H)}.
 \label{eq:gcd-target-intro}
\end{equation}
If the deletion order is $r\mid s$, meaning that a layer at prime $r$ is removed before a layer at prime $s$, the best capacity for that order is
\begin{equation}
 \min\left\{\frac{j_r}{j_r+1},\ \delta
 \max_{\substack{Q\mid C\\Q\le C/s}}
 \frac{d(gQ)}{d(g)}\right\}.
 \label{eq:l2-preview}
\end{equation}
Thus the scheduling problem becomes a divisor-selection problem in the target-only factor $C$.  This formula is the technical hinge of the paper.  It isolates the repair that must be accumulated before the second deletion and makes it possible to certify the half barrier without scanning all paths.

The remaining sections develop this arc.  \Cref{sec:framework} defines the state space, capacity, deletion distance, and exact recursion.  \Cref{sec:static-overlap} disposes of the false gcd surrogate.  \Cref{sec:record-box} proves the reflection and gap theorems.  \Cref{sec:small-deletion} solves the strata $L_-\le2$ up to explicit arithmetic repair conditions.  \Cref{sec:half-certificates} consolidates the principal half-capacity certificates.  \Cref{sec:scheduling} translates the moves into prime layers and proves an insertion-eager normal form.  \Cref{sec:computation} records the finite experiment and its reproducibility boundary.  The auxiliary repair-floor, frontier, scheduling, and certificate-classification details are deferred to the appendices.

\section{Prime-exponent state space and maximin capacity}
\label{sec:framework}

\subsection{Records and exponent profiles}

For an integer $n\ge1$, write
\begin{equation}
 n=\prod_{p}p^{a_p},
 \qquad
 a_p=\vpp pn,
 \qquad
 d(n)=\prod_p(a_p+1),
 \label{eq:prime-profile}
\end{equation}
where only finitely many $a_p$ are nonzero.  We identify $n$ with the finite-support vector $a(n)=(a_p)_p$.  The following standard observations explain why this coordinate system is natural for highly composite numbers.

\begin{lemma}[Basic shape of a divisor record]
\label{lem:basic-hcn-shape}
If $H=\prod_p p^{a_p}$ is highly composite, then its prime support is an initial segment and the exponents are non-increasing with the primes: if $p<q$, then $a_p\ge a_q$.  In particular, after listing the primes as $p_1<p_2<\cdots<p_k$, one may write
\begin{equation}
 H=p_1^{a_1}\cdots p_k^{a_k},
 \qquad
 a_1\ge a_2\ge\cdots\ge a_k\ge1.
 \label{eq:hcn-profile-shape}
\end{equation}
\end{lemma}

\begin{proof}
If $q\mid H$ but a smaller prime $p$ does not divide $H$, then $Hp/q<H$ and $d(Hp/q)=d(H)$, contradicting strict recordhood.  Thus the support is an initial prime segment.  If $p<q$ but $a_p<a_q$, then
\[
 M=H\left(\frac pq\right)^{a_q-a_p}
\]
is an integer smaller than $H$ with the same multiset of exponents and hence $d(M)=d(H)$, again a contradiction.
\end{proof}

We use the strict version of recordhood throughout.  This ensures that $d(m)<d(H)$ for every $m<H$, a fact that will repeatedly convert size inequalities into strict divisor inequalities.

\subsection{Geodesics under a record ceiling}

Let
\[
 m=\prod_{i=1}^k p_i^{a_i},
 \qquad
 n=\prod_{i=1}^k p_i^{b_i},
\]
where zero exponents are appended as needed.  Define
\begin{equation}
 \delta_i=|b_i-a_i|,
 \qquad
 \sigma_i=\operatorname{sgn}(b_i-a_i).
 \label{eq:delta-sigma}
\end{equation}
A directed exponent-lattice geodesic from $m$ to $n$ changes one coordinate by one unit at every step and changes each coordinate only in the direction of its target.  Equivalently, it has the minimum possible length
\begin{equation}
 L(m,n)=\sum_i|b_i-a_i|.
 \label{eq:l1-length}
\end{equation}
Such a path is encoded by states
\begin{equation}
 r=(r_1,\ldots,r_k),
 \qquad 0\le r_i\le\delta_i,
 \label{eq:box-state}
\end{equation}
with associated integer and divisor count
\begin{equation}
 N(r)=\prod_{i=1}^k p_i^{a_i+\sigma_i r_i},
 \qquad
 D(r)=\prod_{i=1}^k(a_i+\sigma_i r_i+1).
 \label{eq:state-ND}
\end{equation}
The state set is the finite exponent box between $a$ and $b$.

\begin{definition}[Ceiling-admissible geodesic and capacity]
\label{def:capacity}
Let $X\ge\max\{m,n\}$.  A geodesic $\gamma:m\leadsto n$ is \emph{$X$-admissible} if every vertex $z\in\gamma$ satisfies $z\le X$.  Its divisor bottleneck is
\begin{equation}
 \operatorname{cap}(\gamma)=\min_{z\in\gamma}d(z).
 \label{eq:path-capacity}
\end{equation}
The optimal geodesic bottleneck is
\begin{equation}
 \Gcal_X(m,n)=
 \max_{\substack{\gamma:m\leadsto n\text{ geodesic}\\z\le X\text{ for all }z\in\gamma}}
 \min_{z\in\gamma}d(z).
 \label{eq:optimal-capacity}
\end{equation}
For consecutive highly composite numbers $H<H'$ we set
\begin{equation}
 \betageo(H,H'):=\frac{\Gcal_{H'}(H,H')}{d(H)}.
 \label{eq:beta-definition}
\end{equation}
\end{definition}

The normalisation makes $\betageo$ a retained fraction of the initial divisor record.  Since both endpoints are admissible and $d(H')>d(H)$, one always has $0<\betageo\le1$.

Separate the required moves into deletions and insertions:
\begin{equation}
 L_-(m,n):=\sum_p(a_p-b_p)_+,
 \qquad
 L_+(m,n):=\sum_p(b_p-a_p)_+.
 \label{eq:Lpm}
\end{equation}
Thus $L=L_-+L_+$.  Each deletion $j\to j-1$ multiplies $d$ by $j/(j+1)$, while each insertion $j-1\to j$ multiplies it by $(j+1)/j$.  The size multiplier is respectively $1/p$ or $p$.  The path problem is therefore a scheduling problem with exact multiplicative weights.

\subsection{Exact recursion}

The exponent box is an acyclic directed graph when states are ordered by $|r|_1$.  The widest-path recurrence is consequently exact and elementary.

\begin{theorem}[Exact maximin recursion]
\label{thm:dp}
 Fix a ceiling $X\ge\max\{N(0),N(\delta)\}=\max\{m,n\}$.  Put $F(r)=0$ for every inadmissible state.  For admissible states define $F(r)$ recursively by
\begin{equation}
 F(0)=D(0)
 \label{eq:dp-initial}
\end{equation}
and, for $r\ne0$,
\begin{equation}
 \boxed{
 F(r)=\min\left\{D(r),
 \max_{\substack{i:r_i>0\\N(r-e_i)\le X}}F(r-e_i)
 \right\}.}
 \label{eq:dp-recursion}
\end{equation}
 Here $\max\varnothing:=0$; equivalently, any state with no reachable admissible predecessor has value $F(r)=0$.  Then
\begin{equation}
 F(\delta)=\Gcal_X(m,n).
 \label{eq:dp-exact}
\end{equation}
The number of states before ceiling exclusions is $\prod_i(\delta_i+1)$.
\end{theorem}

\begin{proof}
Every directed geodesic reaching a nonzero state $r$ arrives from a predecessor $r-e_i$ with $r_i>0$.  If the best bottleneck to that predecessor is $F(r-e_i)$, appending $r$ changes the path minimum to
\[
 \min\{F(r-e_i),D(r)\}.
\]
If every admissible predecessor is unreachable, their values are all zero and the convention $\max\varnothing=0$ (or the maximum of these zero values) gives $F(r)=0$.  Otherwise, maximising over admissible predecessors gives \eqref{eq:dp-recursion}.  Induction on $|r|_1$ proves that $F(r)$ is the best bottleneck among all admissible directed paths from $0$ to $r$, and $r=\delta$ gives the claim.
\end{proof}

\begin{remark}[What the recursion does and does not prove]
The recursion gives a complete exact answer for any fixed pair and ceiling.  It does not by itself explain why a uniform lower bound should hold along the sequence of records.  The rest of the paper extracts arithmetic structure from the box so that many capacities can be certified without enumerating their paths.
\end{remark}

\subsection{The gcd corner and deletion distance}

Let $g=\gcd(m,n)$.  The path that deletes from $m$ to $g$ and then inserts from $g$ to $n$ is always available.

\begin{proposition}[The gcd-route certificate]
\label{prop:gcd-route}
For $m<n$, the delete-then-insert geodesic is $n$-admissible and has bottleneck exactly $d(g)$.  Hence
\begin{equation}
 \Gcal_n(m,n)\ge d(g),
 \qquad
 \frac{\Gcal_n(m,n)}{d(m)}\ge\frac{d(g)}{d(m)}.
 \label{eq:gcd-route}
\end{equation}
\end{proposition}

\begin{proof}
Every deletion state divides $m$ and is therefore below $n$; its divisor count decreases monotonically until $g$.  Every insertion state after $g$ divides $n$, remains below the ceiling, and has increasing divisor count.  The minimum is exactly $d(g)$.
\end{proof}

The gcd overlap can be bounded solely in terms of the number of deleted layers.

\begin{lemma}[Overlap versus deletion distance]
\label{lem:deletion-overlap}
For positive integers $m=\prod p^{a_p}$ and $n=\prod p^{b_p}$,
\begin{equation}
 \frac{d(\gcd(m,n))}{d(m)}
 =\prod_{a_p>b_p}\frac{b_p+1}{a_p+1}
 \ge \prod_{a_p>b_p}\frac1{a_p-b_p+1}
 \ge 2^{-L_-(m,n)}.
 \label{eq:deletion-overlap}
\end{equation}
\end{lemma}

\begin{proof}
Write $a_p=b_p+r$ with $r\ge1$.  Then
\[
 \frac{b_p+1}{a_p+1}=\frac{b_p+1}{b_p+r+1}
 \ge\frac1{r+1}\ge2^{-r},
\]
because $r+1\le2^r$.  Multiplication over decreasing coordinates proves the result.
\end{proof}

A uniform bound on $L_-$ would therefore imply a positive uniform gcd-overlap constant.  The computations in \Cref{sec:computation} give $L_-\le5$ through $10^{70}$, but no universal bound is assumed or proved here.

\section{Static overlap fails}
\label{sec:static-overlap}

The gcd route suggests a simple replacement for the path problem: perhaps consecutive records always retain at least half of the old divisor count at their common coordinatewise corner.  This is false.

\begin{proposition}[Computer-assisted: first failure of the half-gcd bound]
\label{prop:first-half-gcd-counterexample}
The first counterexample in the increasing sequence of consecutive highly composite numbers is the pair
\begin{align*}
 H&=48\,886\,437\,600
   =2^5 3^3 5^2 7^2 11\,13\,17\,19,\\
 H'&=64\,250\,746\,560
   =2^6 3^3 5\,7\,11\,13\,17\,19\,23.
\end{align*}
It satisfies
\begin{equation}
 d(H)=3456,
 \qquad
 d(H')=3584,
 \qquad
 d(\gcd(H,H'))=1536,
 \label{eq:first-half-gcd-data}
\end{equation}
and hence
\begin{equation}
 \frac{d(\gcd(H,H'))}{d(H)}=\frac49<\frac12.
 \label{eq:first-half-gcd-ratio}
\end{equation}
\end{proposition}

\begin{proof}
The displayed factorisations give
\[
 \gcd(H,H')=2^5 3^3 5\,7\,11\,13\,17\,19.
\]
The identity $d(\prod p^{e_p})=\prod(e_p+1)$ gives the three divisor counts, so these parts are elementary checks.  The record-theoretic assertions are computer-assisted.  The complete profile-minimum scan of \Cref{lem:enumeration-completeness}, run through $10^{70}$ as described in \Cref{sec:computation}, places $H$ and $H'$ consecutively in the strict record table and finds no earlier consecutive record pair violating the half-gcd bound.  Completeness of that scan therefore certifies both consecutivity and the word ``first''.
\end{proof}

The scan is not a search over arbitrary integers one by one; it enumerates all non-increasing prime-exponent profiles whose integer value is below the chosen bound and updates the strict record whenever a larger divisor count appears.

\begin{computation}[Static overlap through $10^{70}$]
\label{comp:gcd-stress}
The exact record table contains $889$ records and $888$ consecutive transitions.  Among them, $119$ satisfy
\[
 d(\gcd(H,H'))<\frac12d(H).
\]
The smallest observed ratio is
\begin{equation}
 \boxed{\min\frac{d(\gcd(H,H'))}{d(H)}=\frac{128}{405}.}
 \label{eq:min-gcd-ratio}
\end{equation}
The deletion distance never exceeds $5$ in this range.
\end{computation}

\begin{table}[t]
\centering
\caption{Static gcd overlap in the exact record table.  The column ``violations'' counts consecutive transitions with $d(\gcd(H,H'))<d(H)/2$.}
\label{tab:gcd-stress}
\small
\begin{tabular}{@{}rrrrr@{}}
\toprule
upper bound & records & transitions & violations & minimum ratio \\
\midrule
$10^6$    & 38  & 37  & 0   & $1/2$ \\
$10^{12}$ & 95  & 94  & 4   & $5/14$ \\
$10^{20}$ & 177 & 176 & 4   & $5/14$ \\
$10^{30}$ & 292 & 291 & 17  & $5/14$ \\
$10^{40}$ & 423 & 422 & 33  & $128/405$ \\
$10^{50}$ & 561 & 560 & 38  & $128/405$ \\
$10^{60}$ & 723 & 722 & 56  & $128/405$ \\
$10^{70}$ & 889 & 888 & 119 & $128/405$ \\
\bottomrule
\end{tabular}
\end{table}

The static failures are not marginal numerical noise.  They persist over many scales and reach roughly $0.316$.  At the same time, every computed dynamic capacity in this finite table is at least one half.  The correct object must therefore remember the order in which layers are exchanged.

For later comparison we record the exact primewise identity
\begin{equation}
 d(\gcd(m,n))\,d(\operatorname{lcm}(m,n))=d(m)d(n).
 \label{eq:gcd-lcm-identity}
\end{equation}
Indeed, for exponents $a,b$ one has
\[
 (\min\{a,b\}+1)(\max\{a,b\}+1)=(a+1)(b+1).
\]
The identity highlights why the gcd corner is only one member of a complementary pair.  The exponent box contains many other such pairs, and their divisor product is even larger.  That inequality drives the next section.

\section{Record-box duality and tunnelling}
\label{sec:record-box}

For two integers
\[
 m=\prod_p p^{a_p},
 \qquad
 n=\prod_p p^{b_p},
\]
define their closed exponent box by
\begin{equation}
 \Bcal(m,n):=
 \left\{\prod_p p^{e_p}:\min(a_p,b_p)\le e_p\le\max(a_p,b_p)\right\}.
 \label{eq:record-box-definition}
\end{equation}
Every exponent-lattice geodesic from $m$ to $n$ is contained in this finite set.

\subsection{Complementary states}

\begin{proposition}[Complementary-state reflection]
\label{prop:box-reflection}
For $z\in\Bcal(m,n)$, define
\begin{equation}
 z^\sharp:=\frac{mn}{z}.
 \label{eq:sharp-definition}
\end{equation}
Then $z^\sharp\in\Bcal(m,n)$, $(z^\sharp)^\sharp=z$, and
\begin{equation}
 \log z+\log z^\sharp=\log m+\log n.
 \label{eq:log-reflection}
\end{equation}
Moreover,
\begin{equation}
 \boxed{d(z)d(z^\sharp)\ge d(m)d(n).}
 \label{eq:divisor-reflection}
\end{equation}
\end{proposition}

\begin{proof}
Work prime by prime.  Suppose $a\le b$ and write the exponent of $p$ in $z$ as $e=a+r$ with $0\le r\le b-a$.  The reflected exponent is $a+b-e=b-r$, which remains in $[a,b]$.  Furthermore,
\begin{align}
 &(e+1)(a+b-e+1)-(a+1)(b+1)\nonumber\\
 &\hspace{35mm}=r\bigl((b-a)-r\bigr)\ge0.
 \label{eq:prime-concavity}
\end{align}
Multiplying \eqref{eq:prime-concavity} over all primes gives \eqref{eq:divisor-reflection}.  The other assertions follow directly from $z^\sharp=mn/z$.
\end{proof}

The primewise calculation is a discrete concavity statement: for a fixed sum of two exponents, the product of the shifted exponents is smallest at the endpoints of the interval.  The box reflection preserves total logarithmic size in pairs while increasing, or at least preserving, the paired divisor product.

\subsection{The gap between consecutive records}

\begin{theorem}[Record-box gap]
\label{thm:record-box-gap}
Let $H<H'$ be consecutive highly composite numbers.  Then
\begin{equation}
 \boxed{\Bcal(H,H')\cap(H,H')=\varnothing.}
 \label{eq:record-box-gap}
\end{equation}
Thus no exponent profile coordinatewise between the two records has numerical size strictly between them.
\end{theorem}

\begin{proof}
Assume that $H<z<H'$ for some $z\in\Bcal(H,H')$.  By \Cref{prop:box-reflection}, the complement
\[
 z^\sharp=\frac{HH'}{z}
\]
also belongs to the box.  Since $H<z<H'$, multiplication of the inequalities by $HH'/z$ gives $H<z^\sharp<H'$.  Consecutivity implies
\[
 d(z)\le d(H),
 \qquad
 d(z^\sharp)\le d(H).
\]
Hence $d(z)d(z^\sharp)\le d(H)^2$.  On the other hand, \eqref{eq:divisor-reflection} gives
\[
 d(z)d(z^\sharp)\ge d(H)d(H')>d(H)^2,
\]
a contradiction.
\end{proof}

\begin{corollary}[Record-to-record tunnelling]
\label{cor:tunnelling}
Let $\gamma$ be an $H'$-admissible exponent-lattice geodesic from consecutive highly composite numbers $H<H'$.  Every interior vertex satisfies
\begin{equation}
 z<H,
 \qquad
 d(z)<d(H).
 \label{eq:tunnel-below}
\end{equation}
If the transition contains both deletions and insertions, the first move is a deletion and the last move is an insertion.
\end{corollary}

\begin{proof}
Every geodesic state belongs to $\Bcal(H,H')$.  The ceiling excludes states above $H'$, and \Cref{thm:record-box-gap} excludes states in $(H,H')$.  By unique factorisation, a directed geodesic reaches the endpoint vectors only at its endpoints.  Thus every interior state is below $H$, and strict recordhood gives $d(z)<d(H)$.

In a mixed transition, an initial insertion would create an interior state larger than $H$, impossible by the first part.  If the final move were a deletion, its predecessor would be a prime multiple of $H'$ and would exceed the ceiling.  Hence the final move is an insertion.
\end{proof}

The corollary is the geometric core of the paper.  A mixed transition is forced to leave the old record through a downward move, remain under the old record wall, and return to the new record only at the end.  This immediately turns compulsory deletion layers into upper bounds for the best possible bottleneck.

\begin{proposition}[Mandatory-deletion upper bound]
\label{prop:mandatory-deletion-upper}
Let
\[
 H=\prod p^{a_p}<H'=\prod p^{b_p}
\]
be consecutive highly composite numbers, and suppose that at least one coordinate decreases.  Then
\begin{equation}
 \boxed{
 \betageo(H,H')\le
 \min_{p:a_p>b_p}\frac{b_p+1}{b_p+2}.}
 \label{eq:mandatory-deletion-upper}
\end{equation}
In particular, if a prime disappears from the support, so that $a_p>0=b_p$, then
\begin{equation}
 \betageo(H,H')\le\frac12.
 \label{eq:support-loss-upper}
\end{equation}
\end{proposition}

\begin{proof}
Fix a decreasing prime $p$.  Along every geodesic there is a last deletion in that coordinate, from exponent $b_p+1$ to $b_p$.  Immediately before this deletion the state is either $H$ or an interior admissible box state, so \Cref{cor:tunnelling} gives divisor count at most $d(H)$.  The deletion multiplies the count by $(b_p+1)/(b_p+2)$.  Hence every path has a vertex with divisor count at most that factor times $d(H)$.  Take the minimum over decreasing coordinates.
\end{proof}

\begin{corollary}[Vanishing-prime obstruction]
\label{cor:vanishing-prime-obstruction}
Suppose an $H'$-admissible geodesic from $H$ to $H'$ has at least $d(H)/2$ divisors at every state and a prime $p$ disappears from the support.  Then $a_p=1$, and the move $p^1\to p^0$ is the first move.  In particular, no half-safe geodesic can accommodate two disappearing prime coordinates.
\end{corollary}

\begin{proof}
The last deletion in the $p$-coordinate is $1\to0$ and halves the divisor count.  For the state after that deletion to have at least $d(H)/2$ divisors, the preceding state must have at least $d(H)$ divisors.  The predecessor cannot be $H'$, because its $p$-exponent is $0$, whereas the predecessor has $p$-exponent $1$; hence it is either $H$ or an interior vertex.  By \Cref{cor:tunnelling}, every interior vertex has fewer than $d(H)$ divisors.  Therefore the predecessor is $H$, so this $1\to0$ deletion is the first move.  There can have been no earlier deletion in the same coordinate, and consequently $a_p=1$.  Two disappearing primes would require two distinct first moves.
\end{proof}

The upper bound explains why one half is a natural boundary value: a disappearing exponent-one support prime creates an unavoidable factor $1/2$.  The harder direction is to show that the remaining moves can be scheduled so that no deeper loss occurs.

\section{Exact small-deletion strata}
\label{sec:small-deletion}

The tunnelling theorem makes the first deletion visible and often determines the whole path.  This section solves the cases $L_-\le1$ exactly and gives an exact reduction for $L_-=2$.

\subsection{Zero and one deletion}

\begin{theorem}[Exact capacity for $L_-\le1$]
\label{thm:one-deletion}
Let $H<H'$ be consecutive highly composite numbers.

\begin{enumerate}[label=\textup{(\roman*)}]
\item If $L_-(H,H')=0$, then the transition consists of a single insertion and
\begin{equation}
 \betageo(H,H')=1.
 \label{eq:zero-deletion-capacity}
\end{equation}

\item If $L_-(H,H')=1$, let $p$ be the unique decreasing coordinate and write
\[
 a_p=b_p+1=:a.
\]
Then
\begin{equation}
 \boxed{\betageo(H,H')=\frac{a}{a+1}\ge\frac12.}
 \label{eq:one-deletion-capacity}
\end{equation}
\end{enumerate}
\end{theorem}

\begin{proof}
If $L_-=0$, all moves are insertions.  Any nontrivial partial insertion state lies in the exponent box, is larger than $H$, and is at most $H'$.  The record-box gap excludes such an interior state.  Hence there is exactly one insertion, along which the divisor count increases, proving \eqref{eq:zero-deletion-capacity}.

Now suppose $L_-=1$.  A mixed geodesic must begin with its unique deletion by \Cref{cor:tunnelling}.  The move $a\to a-1$ multiplies the divisor count by $a/(a+1)$.  Every remaining move is an insertion.  Any partial collection of these insertions produces a divisor of $H'$, hence remains below the ceiling, and each insertion increases $d$.  The first post-deletion state is therefore the unique bottleneck.  This gives \eqref{eq:one-deletion-capacity}, which also saturates \Cref{prop:mandatory-deletion-upper}.
\end{proof}

This already covers $392$ of the $888$ transitions in the finite table: seven have $L_-=0$ and $385$ have $L_-=1$.

\subsection{A barrier at the top of the target}

The final insertion is constrained not only by the box but by the fact that every proper divisor of the new record lies below the old record level.

\begin{proposition}[Consecutive-record predecessor barrier]
\label{prop:predecessor-barrier}
Let $H<H'$ be consecutive highly composite numbers, write
\[
 H'=\prod_p p^{b_p},
 \qquad
 J:=\frac{d(H')}{d(H)}>1.
\]
Every proper divisor $M\mid H'$ satisfies
\begin{equation}
 d(M)\le d(H).
 \label{eq:proper-divisor-barrier}
\end{equation}
Consequently, for every $p\mid H'$,
\begin{equation}
 \boxed{J\le\frac{b_p+1}{b_p}.}
 \label{eq:record-jump-bound}
\end{equation}
Every required insertion in the $p$-coordinate has divisor multiplier at least $J$.
\end{proposition}

\begin{proof}
A proper divisor $M$ is smaller than $H'$.  If $M<H$, strict recordhood gives $d(M)<d(H)$; if $M=H$, equality holds; and if $H<M<H'$, consecutivity gives $d(M)\le d(H)$.  Taking $M=H'/p$ yields
\[
 d(H')\frac{b_p}{b_p+1}\le d(H),
\]
which is \eqref{eq:record-jump-bound}.  Since $(j+1)/j$ decreases with $j$, every insertion layer ending at $j\le b_p$ has multiplier at least $(b_p+1)/b_p\ge J$.
\end{proof}

\subsection{Exact two-deletion reduction}

Assume now that $L_-(H,H')=2$.  Let
\begin{equation}
 g=\gcd(H,H'),
 \qquad
 C=\frac{H'}g,
 \qquad
 \delta=\frac{d(g)}{d(H)}.
 \label{eq:l2-gcd-data}
\end{equation}
Then $H/g$ is a product of two primes counted with multiplicity.  The target-only factor $C$ stores all required insertions.

\begin{definition}[Target repair function]
\label{def:target-repair}
For a prime $s$ occurring as the base of a required deletion layer, define
\begin{equation}
 \boxed{
 R_s(C;g):=
 \max_{\substack{Q\mid C\\Q\le C/s}}
 \Gamma_g(Q),
 \qquad
 \Gamma_g(Q):=\frac{d(gQ)}{d(g)}.}
 \label{eq:target-repair-function}
\end{equation}
 Writing $H/g=rs$, where $r$ denotes the other deleted prime factor (possibly $r=s$), the inequality $H=grs<H'=gC$ gives $C/s>r\ge1$.  Hence $Q=1$ is always feasible and the set over which the maximum is taken is nonempty.
 The divisor $Q$ represents the target layers inserted before the deletion at $s$; the inequality $Q\le C/s$ is precisely the ceiling constraint at that moment.
\end{definition}

\begin{theorem}[Exact two-deletion reduction]
\label{thm:l2-reduction}
Let the two required deletion layers have prime bases $r,s$ and old exponent levels $j_r,j_s$.  Consider an admissible order $r\mid s$, meaning that the $r$-layer is deleted first and the $s$-layer second.  Put
\begin{equation}
 \alpha_r=\frac{j_r}{j_r+1}.
 \label{eq:first-deletion-factor}
\end{equation}
The optimal normalised capacity among geodesics with this deletion order is
\begin{equation}
 \boxed{
 \beta_{r\mid s}=
 \min\left\{\alpha_r,\ \delta R_s(C;g)\right\}.}
 \label{eq:l2-order-capacity}
\end{equation}
If the deletions lie in distinct prime coordinates, both orders are admissible and
\begin{equation}
 \boxed{
 \betageo(H,H')=
 \max\{\beta_{r\mid s},\beta_{s\mid r}\}.}
 \label{eq:l2-full-capacity}
\end{equation}
If both deletions occur in one coordinate, only the top-layer-first order is admissible.
\end{theorem}

\begin{proof}
Fix the order $r\mid s$.  Since $H/g$ consists of exactly the two deleted prime factors, counting multiplicity,
\[
 H=grs,
 \qquad
 H'=gC.
\]
By tunnelling, the first move is a deletion, so the state becomes $H/r=gs$.  Before the second deletion the path may insert a stack-compatible partial target factor $Q\mid C$.  The resulting state $gsQ$ satisfies the ceiling precisely when
\[
 gsQ\le gC
 \quad\Longleftrightarrow\quad
 Q\le C/s.
\]
All intermediate states during the construction of $Q$ are smaller than $gsQ$ and hence admissible.  Immediately after the first deletion the normalised divisor level is $\alpha_r$.  Immediately after the second deletion the state is $gQ$, with normalised divisor level
\[
 \frac{d(gQ)}{d(H)}=\delta\Gamma_g(Q).
\]
Between the deletions, and after the second deletion, all moves are insertions and therefore increase the divisor count.  Thus the bottleneck for a fixed $Q$ is exactly
\[
 \min\{\alpha_r,\delta\Gamma_g(Q)\}.
\]
Maximising over feasible $Q$ gives \eqref{eq:l2-order-capacity}.  Maximising over admissible deletion orders gives \eqref{eq:l2-full-capacity}.
\end{proof}

\begin{corollary}[Exact half criterion for two deletions]
\label{cor:l2-half-criterion}
A half-safe two-deletion geodesic exists if and only if at least one admissible order $r\mid s$ satisfies
\begin{equation}
 \boxed{R_s(C;g)\ge\frac{1}{2\delta}.}
 \label{eq:l2-half-criterion}
\end{equation}
\end{corollary}

\begin{proof}
Every deletion multiplier $j/(j+1)$ is at least $1/2$, so the first term in \eqref{eq:l2-order-capacity} cannot obstruct the half bound.  The second term gives exactly \eqref{eq:l2-half-criterion}.
\end{proof}

The formula separates two issues.  The first deletion fixes an unavoidable local loss.  The function $R_s$ measures how much target entropy can be loaded before the second deletion without crossing the ceiling.  No other feature of the path matters.

\begin{lemma}[Hard two-deletion patterns]
\label{lem:hard-l2-patterns}
Let $j_1,j_2\ge1$ be the old exponent levels of the two deletions.  Then
\begin{equation}
 \delta=\frac{j_1}{j_1+1}\frac{j_2}{j_2+1}.
 \label{eq:l2-delta-levels}
\end{equation}
If $\delta<1/2$, then either one level equals $1$, or
\begin{equation}
 (j_1,j_2)=(2,2)
 \label{eq:hard-22}
\end{equation}
up to order.
\end{lemma}

\begin{proof}
Assume $2\le j_1\le j_2$.  If $j_2\ge3$, then
\[
 \delta\ge\frac23\frac34=\frac12.
\]
Hence strict failure of the gcd half certificate with both levels at least two forces $j_1=j_2=2$.
\end{proof}

Thus the genuinely difficult two-deletion cases are sharply constrained before the insertion primes are examined: they contain a support-level deletion $1\to0$, except for the isolated pattern $(2,2)$.

\section{Half-capacity certificates}
\label{sec:half-certificates}

The exact reduction becomes especially transparent when one target insertion lies at or to the right of the second deletion prime.  Such a layer can be postponed to create enough multiplicative room for the deletion while sacrificing at most a factor two in divisor gain.

\subsection{A universal crossing certificate}

\begin{theorem}[Prime-crossing half certificate]
\label{thm:prime-crossing}
Assume $L_-(H,H')=2$ and fix an admissible deletion order $r\mid s$.  If at least one target-only insertion has prime base $q\ge s$, then
\begin{equation}
 \boxed{\betageo(H,H')\ge\frac12.}
 \label{eq:prime-crossing-half}
\end{equation}
If the first deletion is a support deletion $r^1\to r^0$, then
\begin{equation}
 \boxed{\betageo(H,H')=\frac12.}
 \label{eq:prime-crossing-equality}
\end{equation}
\end{theorem}

\begin{proof}
Postpone the top target-only layer at $q$ and put $Q=C/q$.  Then $Q\mid C$ and, because $q\ge s$,
\[
 Q=\frac Cq\le\frac Cs.
\]
Thus $Q$ is feasible in \eqref{eq:target-repair-function}.  Let $\eta$ be the divisor multiplier of the postponed layer.  Since every insertion multiplier is at most $2$,
\[
 \Gamma_g(Q)=\frac{\Gamma_g(C)}{\eta}
 \ge\frac12\Gamma_g(C).
\]
Moreover,
\[
 \delta\Gamma_g(C)=\frac{d(H')}{d(H)}>1.
\]
Hence $\delta\Gamma_g(Q)>1/2$, and \Cref{cor:l2-half-criterion} proves the lower bound.  If the first deletion is $1\to0$, \Cref{prop:mandatory-deletion-upper} gives the matching upper bound.
\end{proof}

The condition is purely positional.  It does not depend on the exponent level of the crossing insertion, because its entropy multiplier is always between $1$ and $2$.

\subsection{Water filling on the source profile}

One may force a crossing without inspecting the target factorisation.  The source profile itself limits the entropy that can be packed entirely to the left of a deletion threshold.

\begin{definition}[Left packing functional]
\label{def:left-packing}
Let $H=\prod_p p^{a_p}$, let $t$ be prime, and let $m\ge0$.  Define
\begin{equation}
 \boxed{
 G_H(t,m):=
 \max_{\substack{x_q\in\N_0,\ q<t\\\sum_{q<t}x_q=m}}
 \prod_{q<t}\frac{a_q+x_q+1}{a_q+1}.}
 \label{eq:left-packing-functional}
\end{equation}
If $t=2$ and $m>0$, the feasible set is empty and we put $G_H(2,m)=0$.  The functional ignores target monotonicity and the multiplicative size of the packet; it is therefore an upper bound for the divisor gain of any $m$ insertions supported on primes below $t$.
\end{definition}

\begin{lemma}[Discrete water filling]
\label{lem:water-filling}
Assume that $t>2$ or $m=0$.  Put $c_q=a_q+1$.  A maximiser in \eqref{eq:left-packing-functional} is obtained by repeating $m$ times: increment a coordinate whose current load $c_q+x_q$ is minimal.  Equivalently, the final loads that receive increments are levelled as far as their lower bounds allow.
\end{lemma}

\begin{proof}
The denominators in \eqref{eq:left-packing-functional} are fixed, so it suffices to maximise $\prod_{q<t}(c_q+x_q)$.  Suppose a candidate has $x_q>0$ and final loads $y_q=c_q+x_q$, $y_r=c_r+x_r$ with $y_q\ge y_r+2$.  Moving one increment from $q$ to $r$ multiplies the relevant numerator product by
\[
 \frac{(y_q-1)(y_r+1)}{y_qy_r}
 =1+\frac{y_q-y_r-1}{y_qy_r}>1.
\]
 Thus no maximiser contains an avoidable gap of at least two.  For sufficiency, the marginal increase in the logarithm of the objective when a coordinate of current load $y$ is incremented is
\[
 \log\frac{y+1}{y},
\]
which decreases strictly with $y$.  Thus incrementing a smallest current load chooses a largest available marginal gain.  In each coordinate the successive marginal gains form a decreasing sequence, so the $m$ largest gains over all coordinates automatically form initial segments of those sequences and hence define a feasible allocation.  The greedy rule selects exactly these gains, up to harmless ties, and is therefore globally optimal.
\end{proof}

\begin{theorem}[Source-profile packing certificate]
\label{thm:source-packing}
Let $H<H'$ be consecutive highly composite numbers with $L_-(H,H')=2$.  Put
\[
 g=\gcd(H,H'),
 \qquad
 \delta=\frac{d(g)}{d(H)},
 \qquad
 m=L_+(H,H').
\]
If the deletion layers lie in distinct coordinates, let $t$ be the smaller prime base; if they lie in one coordinate, let $t$ be that base.  If
\begin{equation}
 \boxed{\delta G_H(t,m)\le1,}
 \label{eq:source-packing-condition}
\end{equation}
then a required insertion has prime base $q\ge t$, and therefore
\begin{equation}
 \boxed{\betageo(H,H')\ge\frac12.}
 \label{eq:source-packing-half}
\end{equation}
\end{theorem}

\begin{proof}
Assume that every required insertion has prime base below $t$.  If $x_q=b_q-a_q$ on increasing coordinates, then $\sum_{q<t}x_q=m$, and telescoping the insertion multipliers gives
\[
 \frac{d(H')}{d(H)}
 =\delta\prod_{q<t}\frac{a_q+x_q+1}{a_q+1}
 \le\delta G_H(t,m)\le1.
\]
This contradicts $d(H')>d(H)$.  Hence a crossing insertion $q\ge t$ exists.  Choose the deletion order that leaves the $t$-layer second; \Cref{thm:prime-crossing} applies.
\end{proof}

The theorem replaces target inspection by a finite concave allocation problem on the source exponents.  It is often possible to avoid even the water-filling calculation.

\begin{corollary}[Coarse entropy-budget tests]
\label{cor:coarse-packing}
Assume the hypotheses of \Cref{thm:source-packing} and $\delta<1/2$.

\begin{enumerate}[label=\textup{(\roman*)}]
\item Suppose one deletion is a support loss $P^1\to P^0$ and the other is $t^A\to t^{A-1}$.  If
\begin{equation}
 \frac{A}{2(A+1)}
 \left(\frac{A+2}{A+1}\right)^m\le1,
 \label{eq:coarse-support-test}
\end{equation}
then crossing is forced.

\item In the exceptional non-support pattern $(2,2)$, if
\begin{equation}
 \frac49\left(\frac43\right)^m\le1,
 \label{eq:coarse-22-test}
\end{equation}
then crossing is forced.
\end{enumerate}
\end{corollary}

\begin{proof}
Under failure of crossing, all insertion primes satisfy $q<t$.  In case (i), source monotonicity gives $a_q\ge A$ for such primes, so every insertion multiplier is at most $(A+2)/(A+1)$.  Since $\delta=A/[2(A+1)]$, \eqref{eq:coarse-support-test} implies \eqref{eq:source-packing-condition}.  In case (ii), $a_q\ge2$, every insertion multiplier is at most $4/3$, and $\delta=4/9$.
\end{proof}

\subsection{A universal small-prime support-loss theorem}

A simple exchange against the largest support prime controls the exponent profile near the left edge.

\begin{lemma}[Support-prime exponent floor]
\label{lem:support-prime-floor}
Let $H$ be highly composite, let $P$ be its largest prime factor and assume $P$ occurs to exponent one.  If $q<P$ has exponent $a_q$, then
\begin{equation}
 \boxed{P\le q^{a_q+1}.}
 \label{eq:support-prime-floor}
\end{equation}
\end{lemma}

\begin{proof}
If $q^{a_q+1}<P$, then
\[
 M=H\frac{q^{a_q+1}}P<H.
\]
Removing $P^1$ halves the divisor count, while increasing the $q$-exponent from $a_q$ to $2a_q+1$ doubles it.  Thus $d(M)=d(H)$, contradicting strict recordhood.
\end{proof}

\begin{theorem}[Small-prime support-loss crossing]
\label{thm:small-prime-crossing}
Let $H<H'$ be consecutive highly composite numbers with $L_-=2$.  Suppose one deletion removes the largest support prime $P$ through $P^1\to P^0$, and the other removes one layer at a prime $t\le3$.  Then a target insertion with prime base $q\ge t$ is unavoidable, and
\begin{equation}
 \boxed{\betageo(H,H')=\frac12.}
 \label{eq:small-prime-equality}
\end{equation}
\end{theorem}

\begin{proof}
For $t=2$, any nonempty insertion set automatically contains a prime $q\ge2$.  Since two deletions lower the divisor count and $H'$ is a new record, the insertion set is nonempty.

Let $t=3$, let the deleted layer be $3^A\to3^{A-1}$, and write $a=a_2$.  If crossing failed, all $m$ insertions would occur at $2$, so
\[
 \frac{H'}H=\frac{2^m}{3P}.
\]
Because $d(2H)>d(H)$, the next record satisfies $H'\le2H$, and hence $2^m\le6P$.  By \Cref{lem:support-prime-floor}, $P\le2^{a+1}$, so
\[
 2^m\le6\,2^{a+1}<2^{a+4},
 \qquad m\le a+3.
\]
The divisor jump would then obey
\begin{align*}
 \frac{d(H')}{d(H)}
 &=\frac{A}{2(A+1)}\frac{a+m+1}{a+1}\\
 &\le\frac{A}{A+1}\frac{a+2}{a+1}
 \le\frac{A(A+2)}{(A+1)^2}<1,
\end{align*}
where $a\ge A$ follows from source monotonicity.  This contradicts recordhood.  Crossing is forced, \Cref{thm:prime-crossing} gives the lower bound $1/2$, and the support deletion gives the upper bound.
\end{proof}

\subsection{Matching deletions with compensating insertions}

The same mechanism extends beyond two deletions.  Once an exponent-one support layer has imposed the factor $1/2$, every further deletion can be neutralised by an insertion whose layer level is no higher.

\begin{theorem}[Compensated deletion matching]
\label{thm:compensated-matching}
Let
\[
 H=\prod_p p^{a_p}<H'=\prod_p p^{b_p}
\]
be consecutive highly composite numbers.  Suppose exactly one source support prime disappears: for some prime $r$,
\[
 a_r=1,
 \qquad
 b_r=0,
\]
and no prime $p\ne r$ satisfies $a_p>0=b_p$.  Delete the layer $(r,1)$ first.  List all remaining required deletion layers as
\[
 D_i:(p_i,A_i),\qquad A_i\to A_i-1,
 \qquad 1\le i\le k,
\]
in a stack-admissible order; explicitly, if $p_i=p_h$ and $i<h$, then $A_i>A_h$.  Match them injectively to distinct required insertion layers
\[
 I_i:(q_i,j_i),\qquad j_i-1\to j_i.
\]
Set $x_0=H/r$.  At stage $i$, put $e_i=\vpp{q_i}{x_{i-1}}$.  Require $e_i<j_i$, insert successively the still missing prerequisite layers
\[
 (q_i,e_i+1),\ldots,(q_i,j_i-1)
\]
followed by $I_i=(q_i,j_i)$, and denote the resulting state by
\[
 y_i=x_{i-1}q_i^{j_i-e_i}.
\]
Then execute $D_i$ and put $x_i=y_i/p_i$.  Suppose, for every $i$, that
\begin{enumerate}[label=\textup{(\alph*)}]
\item none of the prerequisite layers is one of the later matched layers $I_{i+1},\ldots,I_k$;
\item $y_i\le H$;
\item $j_i\le A_i$.
\end{enumerate}
After stage $k$, execute all remaining required insertion layers in any stack-admissible order.  Then this schedule is an $H'$-admissible geodesic and
\begin{equation}
 \boxed{\betageo(H,H')=\frac12.}
 \label{eq:matching-equality}
\end{equation}
If $q_i>p_i$, condition (c) follows automatically from the non-increasing target exponent profile.
\end{theorem}

\begin{proof}
The first deletion changes the $r$-exponent from $1$ to $0$, and therefore
\[
 d(x_0)=\frac12d(H).
\]
The deletion order is stack-admissible, and an insertion coordinate cannot be a deletion coordinate; hence $D_i$ is available when it is called.  At stage $i$, all prerequisite insertions increase both the integer and its divisor count.  They end at the state immediately preceding $I_i$, while $y_i$ is the state immediately after $I_i$.  Consequently every state in this insertion block is at most $y_i\le H$.  Also $x_i=y_i/p_i<y_i$, so the state following $D_i$ respects the same ceiling.

Let $u_i$ be the state immediately before $I_i$.  The prerequisite insertions give $d(u_i)\ge d(x_{i-1})$.  Since $I_i$ raises level $j_i-1$ to $j_i$ and $D_i$ lowers level $A_i$ to $A_i-1$,
\[
 d(x_i)
 =d(u_i)\frac{j_i+1}{j_i}\frac{A_i}{A_i+1}
 \ge d(u_i)
 \ge d(x_{i-1}),
\]
where the first inequality is condition~\textup{(c)}.  Induction therefore shows that every state through the last deletion has at least $d(H)/2$ divisors.

Once all deletions have been performed, the current state divides $H'$.  Every subsequent state obtained by inserting a remaining required layer is again a divisor of $H'$, hence is at most $H'$, and its divisor count can only increase.  The completed schedule is thus an $H'$-admissible geodesic with capacity at least $1/2$.  The deletion $r^1\to r^0$ gives the reverse inequality by \Cref{prop:mandatory-deletion-upper}, proving equality.

If $q_i>p_i$, target monotonicity gives
\[
 j_i\le b_{q_i}\le b_{p_i}\le A_i-1<A_i.
\]
Thus condition~\textup{(c)} is automatic.
\end{proof}

The theorem isolates a clean combinatorial target: pair every non-support deletion with a sufficiently low target insertion that fits before it.  In the exact range through $10^{70}$, all $124$ support-loss transitions admit such matchings; this is reported as a finite verification in \Cref{sec:computation}.

The repair-floor and frontier methods give further arithmetic certificates, but are not needed for the logical development of the fixed-endpoint maximin problem.  Their statements and the four explicit second-layer frontier bounds are collected in \Cref{sec:repair-frontier}.

\section{Prime layers and release scheduling}
\label{sec:scheduling}

The previous criteria treat deletions, insertions, and arithmetic competitors directly.  A complementary viewpoint resolves every exponent into unit layers.  It is the most transparent place where the vocabulary of energy and entropy enters: logarithmic prime mass is the constrained resource, while logarithmic divisor count is the quantity whose valley we seek to make shallow.

\subsection{The exact layer spectrum}

For a prime $p$ and an integer $j\ge1$, the layer $(p,j)$ raises the $p$-exponent from $j-1$ to $j$.  Assign to it the size weight, divisor-entropy value, and local slope
\begin{equation}
 w_{p,j}:=\log p,
 \qquad
 s_{p,j}:=\log\frac{j+1}{j},
 \qquad
 \rho_{p,j}:=\frac{s_{p,j}}{w_{p,j}}.
 \label{eq:layer-data}
\end{equation}
For $n=\prod_p p^{a_p}$, let
\begin{equation}
 \Lambda(n):=\{(p,j):1\le j\le a_p\}.
 \label{eq:layer-set}
\end{equation}
Layers in a fixed prime coordinate form a stack: insertions must proceed upward in $j$, and deletions downward.

\begin{proposition}[Exact layer decomposition]
\label{prop:layer-decomposition}
For every $n\ge1$,
\begin{equation}
 \boxed{
 \log n=\sum_{(p,j)\in\Lambda(n)}w_{p,j},
 \qquad
 \log d(n)=\sum_{(p,j)\in\Lambda(n)}s_{p,j}.}
 \label{eq:layer-decomposition}
\end{equation}
Relative to $g=\gcd(m,n)$, a geodesic from $m$ to $n$ deletes exactly the layers in $\Lambda(m)\setminus\Lambda(g)$ and inserts exactly the layers in $\Lambda(n)\setminus\Lambda(g)$.
\end{proposition}

\begin{proof}
The first identity repeats $\log p$ exactly $a_p$ times.  The second telescopes in each prime coordinate:
\[
 \sum_{j=1}^{a_p}\log\frac{j+1}{j}=\log(a_p+1).
\]
Summing over $p$ gives the two formulas.  The statement about a geodesic follows because every coordinate moves monotonically from its source exponent to its target exponent.
\end{proof}

The same slopes are the critical spectrum of superior highly composite numbers.  This makes the scheduling formulation compatible with, rather than external to, the classical theory of Ramanujan and Alaoglu--Erd\H{o}s \citep{Ramanujan1915,AlaogluErdos1944}.

\begin{proposition}[Superior-threshold spectrum]
\label{prop:superior-spectrum}
For $\lambda>0$, consider
\begin{equation}
 \Phi_\lambda(n):=\log d(n)-\lambda\log n.
 \label{eq:phi-lambda}
\end{equation}
Away from a critical tie, the exponent in prime direction $p$ of the unique maximiser of $\Phi_\lambda$ is
\begin{equation}
 a_p(\lambda)
 =\#\{j\ge1:\rho_{p,j}>\lambda\}
 =\left\lfloor\frac{1}{p^\lambda-1}\right\rfloor.
 \label{eq:superior-profile}
\end{equation}
At a tie $\lambda=\rho_{p,j}$, the boundary layer $(p,j)$ has zero $\Phi_\lambda$-increment and may be present or absent without changing the objective.
\end{proposition}

\begin{proof}
By \Cref{prop:layer-decomposition},
\[
 \Phi_\lambda(n)
 =\sum_{(p,j)\in\Lambda(n)}
 \left(\log\frac{j+1}{j}-\lambda\log p\right).
\]
For a fixed $p$ the summands decrease strictly with $j$, so the positive ones form an initial stack.  Positivity is equivalent to
\[
 1+\frac1j>p^\lambda
 \quad\Longleftrightarrow\quad
 j<\frac{1}{p^\lambda-1}.
\]
Away from equality, the number of positive integral $j$ is the floor in \eqref{eq:superior-profile}.  At equality the boundary summand vanishes.
\end{proof}

\subsection{Prefix resources and a local exchange theorem}

Let a geodesic word be $e_1,\ldots,e_L$, and put $\varepsilon_r=+1$ for an insertion and $\varepsilon_r=-1$ for a deletion.  If $w_r,s_r$ are the data of the traversed layer, then at the $k$th vertex $z_k$ one has the exact prefix identities
\begin{align}
 \log\frac{z_k}{H}
 &=\sum_{r\le k}\varepsilon_r w_r,
 \label{eq:prefix-size}\\
 \log\frac{d(z_k)}{d(H)}
 &=\sum_{r\le k}\varepsilon_r s_r.
 \label{eq:prefix-entropy}
\end{align}
The ceiling condition is therefore
\begin{equation}
 \sum_{r\le k}\varepsilon_r w_r
 \le\log\frac{H'}{H}
 \qquad(0\le k\le L).
 \label{eq:prefix-ceiling}
\end{equation}
Thus the transition is an exact two-coordinate prefix-scheduling problem, with one coordinate constrained from above and the other optimised from below.

\begin{lemma}[Insertion-before-deletion exchange]
\label{lem:exchange}
Suppose a ceiling-admissible geodesic contains two adjacent moves in different prime coordinates: first a required deletion, then a required insertion.  If that insertion was already ceiling-admissible before the deletion, swapping the two moves preserves the endpoint and feasibility and strictly raises the minimum divisor count within the two-step segment.
\end{lemma}

\begin{proof}
Let the segment start at $z$, with $d(z)=D$.  Write its deletion and insertion multipliers as $r_-\in(0,1)$ and $r_+>1$.  Before the swap the three divisor counts are
\[
 D,\qquad Dr_-,\qquad Dr_-r_+,
\]
so the local minimum is $Dr_-$.  After the swap they are
\[
 D,\qquad Dr_+,\qquad Dr_+r_-,
\]
and both noninitial values exceed $Dr_-$.  The swapped intermediate integer is admissible by hypothesis; the final integer is unchanged.  Distinct coordinates ensure that stack availability is unaffected.
\end{proof}

\begin{corollary}[Adjacent insertion-eager normal form]
\label{cor:insertion-eager}
Among capacity-maximising geodesics there is one containing no adjacent deletion--insertion pair for which the insertion was already feasible before the deletion.
\end{corollary}

\begin{proof}
For a geodesic word $\gamma=(e_1,\ldots,e_L)$, define its total deletion--insertion inversion count by
\[
 \operatorname{Inv}(\gamma)
 :=\#\{(u,v):1\le u<v\le L,
 e_u\text{ is a deletion and }e_v\text{ is an insertion}\}.
\]
Choose, among all capacity-maximising geodesics, one with minimal $\operatorname{Inv}(\gamma)$.  If it contained an adjacent deletion--insertion pair of the kind described in the statement, \Cref{lem:exchange} would allow the two moves to be swapped without lowering the bottleneck.  The swapped path would therefore still be capacity-maximising.  Because the two moves are adjacent, all their order relations with every other move are unchanged, while their mutual deletion--insertion inversion is removed.  Thus $\operatorname{Inv}$ decreases by exactly one, a contradiction.
\end{proof}

The corollary is intentionally local.  It does not assert that every feasible insertion must globally precede every deletion; intervening stack constraints and the ceiling can matter.  It does, however, justify searching among schedules that consume readily available repair before releasing more prime mass.

\subsection{A record-specific greedy rule}

The finite data suggest the following deterministic scheduler.
\begin{enumerate}[label=\textup{(G\arabic*)}]
\item If at least one required insertion is ceiling-feasible, perform a feasible insertion of largest current slope $\rho_{p,j}$.
\item If no required insertion is feasible, delete the top required layer at the largest available prime.
\end{enumerate}
An independently proved equality of slopes is resolved lexicographically by the smaller prime $p$ and then the smaller layer index $j$.  In the finite checker, unequal slopes are ordered by outward-rounded interval arithmetic; overlapping enclosures trigger a precision increase rather than being treated as a tie.
The rule is not presented as a theorem.  Its exact agreement with dynamic programming for all record transitions in the computed range is reported in \Cref{comp:greedy}.  The record hypothesis is essential even for very orderly exponent profiles.

\begin{proposition}[A nonrecord counterexample to the greedy rule]
\label{prop:greedy-counterexample}
On the primes $2,3,5,7$, let
\begin{equation}
 a=(5,5,4,1),
 \qquad
 b=(5,4,3,3).
 \label{eq:greedy-counterexample-profiles}
\end{equation}
They represent
\[
 34\,020\,000\longrightarrow111\,132\,000,
 \qquad
 360\longrightarrow480
\]
in divisor count.  The exact geodesic bottleneck is $300$, whereas the rule \textup{(G1)--(G2)} produces bottleneck $288$.
\end{proposition}

\begin{proof}
No insertion at $7$ is initially ceiling-feasible, since
$7\cdot34\,020\,000>111\,132\,000$.  Hence every path begins by deleting at $3$ or at $5$.  Those moves leave respectively $300$ and $288$ divisors, so no path can have bottleneck above $300$.  The schedule
\[
 /3,\ \times7,\ /5,\ \times7
\]
is ceiling-admissible and has divisor counts
\[
 360\to300\to450\to360\to480,
\]
proving optimality.  Rule (G2) instead deletes at the larger prime $5$, and its divisor counts begin
\[
 360\to288\to432\to360\to480.
\]
Thus its bottleneck is $288$.
\end{proof}

An abstract resource-balancing example in \Cref{app:generic-scheduling-counterexample} shows why total resource and total divisor entropy alone cannot imply the half barrier: the special layer spectrum, stack precedence, and the record-box gap are genuine arithmetic input.

\section{Exact computation through \texorpdfstring{$10^{70}$}{10 to the 70}}
\label{sec:computation}

The numerical statements in this section are finite computations with explicitly separated arithmetic.  Integer values, divisor counts, gcds, ceiling tests, and maximin bottlenecks are evaluated exactly with integer arithmetic.  Logarithms enter only in ordering the heuristic slopes of (G1), where interval enclosures certify each comparison.  None of the computations is used as a substitute for an unproved universal implication.

\subsection{Complete record enumeration below a bound}

The record scan does not test every integer.  It enumerates the necessary exponent profiles from \Cref{lem:basic-hcn-shape}.

\begin{lemma}[Completeness of the profile-minimum scan]
\label{lem:enumeration-completeness}
Fix $B\ge1$.  Enumerate every finite sequence
\begin{equation}
 a_1\ge a_2\ge\cdots\ge a_k\ge1,
 \qquad
 \prod_{i=1}^k p_i^{a_i}\le B,
 \label{eq:enumerated-profiles}
\end{equation}
including the empty sequence.  For each divisor count $D=\prod_i(a_i+1)$ retain only the least integer realising $D$.  Sort these retained minima by size and keep precisely the strict successive increases of $D$.  The resulting list is the complete list of highly composite numbers at most $B$.
\end{lemma}

\begin{proof}
By \Cref{lem:basic-hcn-shape}, every highly composite number at most $B$ occurs among the enumerated profiles.  A strict divisor record $H$ is necessarily the least integer with divisor count $d(H)$: an earlier integer with the same count would contradict strictness.  Hence no record is lost when one retains only the least representative of each divisor count.  Finally, scanning those representatives in increasing numerical order and keeping strict increases is exactly the definition of the record subsequence.
\end{proof}

 The archived depth-first C++ scanner in the versioned Zenodo reproducibility package (DOI: \zenododoi) generates \eqref{eq:enumerated-profiles} recursively.  It uses exact $256$-bit integers for $n$, stores the minimum representative of each divisor count, and writes each record together with its exponent vector.  The table used below is generated by that scanner; it is not imported from an external record list.

\begin{computation}[Global finite census]
\label{comp:global-census}
For $B=10^{70}$ the scan visits $723\,533\,871$ non-increasing exponent profiles, retains $1\,891\,853$ distinct divisor-count minima, and extracts $889$ strict records, hence $888$ consecutive transitions.  The principal transition statistics are displayed in \Cref{tab:global-census}.
\end{computation}

\begin{table}[t]
\centering
\caption{Exact transition census through $10^{70}$.  Ratios and counts refer only to this finite range.}
\label{tab:global-census}
\begin{tabularx}{0.92\textwidth}{@{}Xr@{}}
\toprule
Quantity & Exact value\\
\midrule
Enumerated profiles & $723\,533\,871$\\
Distinct divisor-count minima & $1\,891\,853$\\
Highly composite numbers & $889$\\
Consecutive transitions & $888$\\
Transitions with $d(\gcd(H,H'))<d(H)/2$ & $119$\\
Minimum static gcd ratio & $128/405$\\
Transitions with $\betageo<1/2$ & $0$\\
Transitions with $\betageo=1/2$ & $124$\\
Maximum deletion distance $L_-$ & $5$\\
Largest formal geodesic box & $128$ states\\
\bottomrule
\end{tabularx}
\end{table}

\subsection{Static overlap versus dynamic capacity}

For each consecutive pair, the exact recursion \eqref{eq:dp-recursion} is evaluated over its full exponent box.  The contrast with the static gcd surrogate is visible in \Cref{fig:transition-ratios}.  The open circles can fall well below one half; the black crosses do not in the tested range.  The distinct markers keep the two series legible in greyscale.

\begin{figure}[t]
\centering
\includegraphics[width=0.94\textwidth]{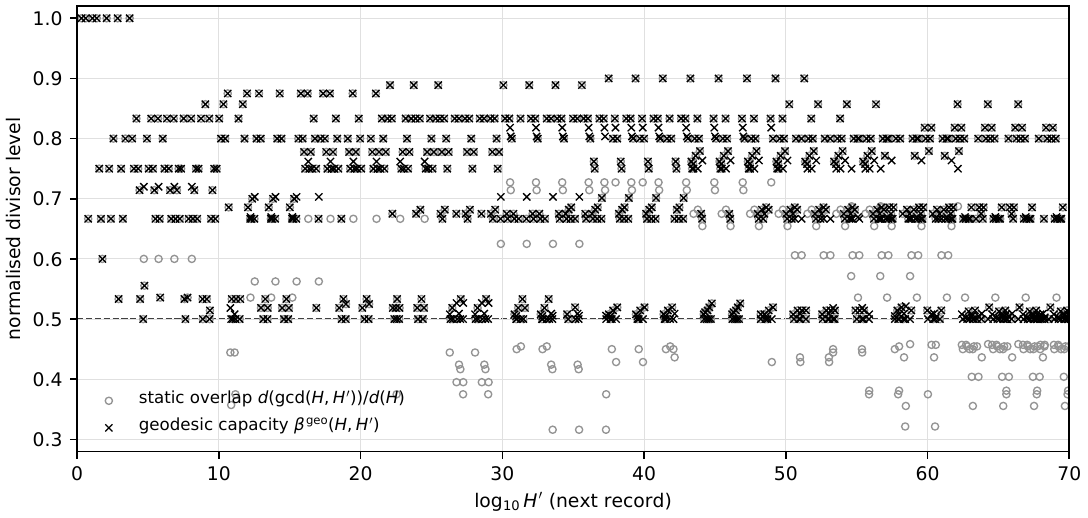}
\caption{Static overlap $d(\gcd(H,H'))/d(H)$ (open circles) and exact dynamic capacity $\betageo(H,H')$ (black crosses) for all $888$ consecutive record transitions through $10^{70}$.  The horizontal coordinate is $\log_{10}H'$.  The dashed line at $1/2$ is a visual reference, not an assumed bound.}
\label{fig:transition-ratios}
\end{figure}

The deletion-distance strata reveal where the half floor is tight.

\begin{table}[t]
\centering
\caption{Exact dynamic capacities stratified by deletion distance.  ``Equality'' counts transitions with capacity $1/2$; the final column is the minimum capacity in that stratum.}
\label{tab:deletion-strata}
\begin{tabular}{@{}rrrr@{}}
\toprule
$L_-$ & Transitions & Equality cases & Minimum $\betageo$\\
\midrule
$0$ & $7$   & $0$  & $1$\\
$1$ & $385$ & $59$ & $1/2$\\
$2$ & $301$ & $61$ & $1/2$\\
$3$ & $151$ & $4$  & $1/2$\\
$4$ & $37$  & $0$  & $1024/2025$\\
$5$ & $7$   & $0$  & $28/55$\\
\bottomrule
\end{tabular}
\end{table}

\begin{computation}[Support loss and equality]
\label{comp:support-equality}
Among the $888$ transitions, the support size changes by $-1,0,+1$ in respectively $124,605,159$ cases.  Exactly the $124$ support-loss transitions have $\betageo=1/2$.  The equality cases split as $59,61,4$ across $L_-=1,2,3$.  Every one admits a compensated matching of \Cref{thm:compensated-matching}; no matching failure occurs.
\end{computation}

This finite equivalence is sharper than the half bound alone.  It also agrees with the unconditional upper obstruction: losing an exponent-one support prime forces the path through a state with at most half the old divisor count.

\paragraph{Independently checked two-deletion certificates.}
The exact reduction of \Cref{thm:l2-reduction} was evaluated independently of the full path recursion.  There were no mismatches, and each of the $65$ gcd-hard transitions received both a prime-crossing and a source-profile packing certificate.  The full classification, hard deletion patterns, and frontier-competition audit are given in \Cref{app:l2-certificate-audit}.

\subsection{Greedy stress test and forced release blocks}

\begin{computation}[Record-specific greedy agreement]
\label{comp:greedy}
For all $888$ transitions, rule \textup{(G1)--(G2)} has exactly the same bottleneck as the dynamic-programming optimum.  There are no mismatches.  The checker made $673$ slope comparisons using arbitrary-precision interval arithmetic, starting at $128$ bits and doubling the precision if two enclosures overlapped.  Every comparison was certified by disjoint enclosures already at $128$ bits; there were no unresolved comparisons or slope equalities.  The smallest point separation between compared slopes was greater than $2.0186\times10^{-4}$.

Along these greedy paths, every maximal consecutive block of forced deletions has length at most $5$, its raw divisor-loss factor is at most $2$, and the state after every such block has at least $d(H)/2$ divisors.  The longest block has five deletions and raw loss $77/40$; the factor $2$ occurs only for a single support deletion.
\end{computation}

The block statistic points to a more local route toward a proof: rather than controlling an arbitrary full schedule, one may try to show that each record-forced release block loses at most one bit of divisor entropy before repair becomes feasible.

\subsection{Reproducibility boundary}

 The arXiv source package contains only files required to compile the article.  A separate reproducibility package is archived at Zenodo as version 1.0.0 (DOI: \zenododoi).  It contains the exact TSV record table, the archived C++ scanner, the full maximin recursion, the interval-certified greedy checker, the figure generator, and separate programs for the record-box structure, exact two-deletion reduction, crossing and source-packing certificates, frontier competition, and compensated support matching.

 The package provides one-command quick, retained-table, and full modes, expected JSON and TSV outputs, execution logs, a pinned Python environment, compiler and interpreter versions, hardware and runtime notes, a licence, and a SHA-256 manifest.  The quick mode rebuilds the record table through $10^{20}$ and runs every checker, the retained-table mode verifies all cited counts from the included $10^{70}$ table without rebuilding it, and the full mode rebuilds that table through $10^{70}$ before rerunning every check.  Integer states, ceiling tests, divisor counts, and bottlenecks use exact arithmetic; only the greedy slope order uses the certified interval procedure described above.

\section{Open problems}
\label{sec:open}

The results above leave a compact set of record-specific questions.  They are stated separately from the theorems and the finite census.

\begin{conjecture}[Geodesic half-entropy conjecture]
\label{conj:geodesic-half}
For every pair $H<H'$ of consecutive highly composite numbers,
\begin{equation}
 \boxed{\betageo(H,H')\ge\frac12.}
 \label{eq:geodesic-half-conjecture}
\end{equation}
Equivalently, there is an $H'$-admissible exponent geodesic along which the divisor entropy drops by at most $\log2$ below $\log d(H)$.
\end{conjecture}

The conjecture is proved here for $L_-\le1$, for every pair satisfying the gcd certificate, and under the crossing, source-packing, or compensated-matching hypotheses above.  The auxiliary repair-floor and frontier results in \Cref{sec:repair-frontier} provide additional arithmetic tests that can force the crossing hypothesis in two-deletion cases.  The exact $L_-=2$ reduction narrows the remaining universal issue to arithmetic control of the divisor-selection functional $R_s(C;g)$.

\begin{conjecture}[Equality classification]
\label{conj:equality-classification}
For consecutive highly composite numbers,
\begin{equation}
 \betageo(H,H')=\frac12
 \quad\Longleftrightarrow\quad
 |\supp(H')|<|\supp(H)|.
 \label{eq:equality-classification}
\end{equation}
In words, equality occurs exactly when at least one exponent-one support prime disappears.
\end{conjecture}

Assuming \Cref{conj:geodesic-half}, the reverse implication in \eqref{eq:equality-classification} follows from the mandatory-deletion upper bound.  The forward implication is subtler: it asks for a strict amount of repair whenever the supports do not shrink.

\begin{conjecture}[Record slope--release scheduler]
\label{conj:greedy}
For consecutive highly composite numbers, rule \textup{(G1)--(G2)} produces a ceiling-admissible geodesic of optimal bottleneck capacity.
\end{conjecture}

The nonrecord example in \Cref{prop:greedy-counterexample} shows that monotone exponent profiles alone are insufficient.  A proof would have to connect the local slope order to record-specific exchange inequalities or to the absence of box states between consecutive records.  A weaker but already decisive target would be to prove that every forced deletion block of this scheduler loses at most a factor $2$ in divisor count before the next insertion becomes feasible.

\section{Conclusion}
\label{sec:conclusion}

The passage between consecutive highly composite numbers has a geometry that is invisible to a static gcd comparison.  The gcd corner can fall below half the old divisor record, yet an interleaved geodesic can repair entropy before the next compulsory deletion.  The complementary involution $z\mapsto HH'/z$ explains why this repair takes place in a sharply constrained region: between consecutive records, the exponent box has no numerical state in the open interval $(H,H')$.  A mixed path must descend below the old record and return only at the new one.

Within that tunnel, the exact maximin recursion provides the ground truth, the two-deletion formula isolates the arithmetic repair functional, and crossing, packing, competitive-repair, frontier, and matching arguments turn local prime information into rigorous lower certificates.  The layer spectrum then reveals the same threshold structure that generates superior highly composite numbers, but now as a scheduling problem with a hard logarithmic-size wall.

The finite census through $10^{70}$ is coherent in a way that is hard to dismiss but easy to state honestly: the static half-gcd conjecture fails $119$ times; the dynamic half barrier does not fail; equality occurs exactly at support loss; and a simple record-specific scheduler reproduces every exact optimum.  The universal theorem remains open.  What the present paper contributes is a tractable state space, several proved mechanisms, an exact two-deletion reduction, and a reproducible set of sharply formulated targets.

\appendix

\section{Anatomy of the first static counterexample}
\label{app:first-counterexample}

For the pair in \eqref{eq:first-counterexample-intro},
\begin{align*}
 H&=2^5 3^3 5^2 7^2 11\,13\,17\,19,\\
 H'&=2^6 3^3 5\,7\,11\,13\,17\,19\,23.
\end{align*}
The gcd and target-only factor are
\begin{equation}
 g=2^5 3^3 5\,7\,11\,13\,17\,19
   =1\,396\,755\,360,
 \qquad
 C=\frac{H'}g=46.
 \label{eq:first-example-gC}
\end{equation}
Since $d(g)=1536$ and $d(H)=3456$,
\begin{equation}
 \delta=\frac{d(g)}{d(H)}=\frac49.
 \label{eq:first-example-delta}
\end{equation}
The two deletion layers are $5^2\to5^1$ and $7^2\to7^1$, each with first-deletion factor $2/3$.  If the $5$-layer is removed first, the repair available before the remaining $7$-deletion is restricted by
\[
 Q\mid46,
 \qquad
 Q\le\frac{46}{7}.
\]
The optimal choice is $Q=2$, which raises the $2$-exponent from $5$ to $6$ and gives gain $7/6$.  Hence \Cref{thm:l2-reduction} yields
\begin{equation}
 \betageo(H,H')
 =\min\left\{\frac23,\frac49\cdot\frac76\right\}
 =\frac{14}{27}.
 \label{eq:first-example-capacity}
\end{equation}
The same value is obtained in the opposite deletion order.  One optimal path is
\[
 H\xrightarrow{/7}\frac H7
 \xrightarrow{\times2}\frac{2H}{7}
 \xrightarrow{/5}\frac{2H}{35}
 \xrightarrow{\times23}H',
\]
with divisor counts
\[
 3456\to2304\to2688\to1792\to3584.
\]
This example is the smallest place where the distinction between overlap and schedule becomes unavoidable: the gcd ratio is $4/9$, while the geodesic capacity is $14/27$.

\section{Pseudocode for the exact bottleneck recursion}
\label{app:dp-pseudocode}

For completeness, \Cref{alg:dp} records the finite calculation behind \Cref{thm:dp}.  States are processed in increasing $\ell^1$ distance from the source, so every predecessor has already been evaluated.

\begin{figure}[H]
\begin{tcolorbox}[colback=softgray,colframe=deepblue!70,boxrule=0.7pt,arc=1.2mm]
\small\ttfamily
Input: source exponents a, target exponents b, ceiling X\par
$\delta_i\gets|b_i-a_i|$; $F\gets$ empty map\par
for $r$ in exponent box, ordered by $|r|_1$:\par
\quad compute exact $N(r)$ and $D(r)$\par
\quad if $N(r)>X$: continue\par
\quad if $r=0$: $F(r)\gets D(r)$; continue\par
\quad $M\gets0$\par
\quad for each $i$ with $r_i>0$ and $r-e_i$ in $F$:\par
\qquad $M\gets\max\{M,F(r-e_i)\}$\par
\quad if $M>0$: $F(r)\gets\min\{D(r),M\}$\par
return $F(\delta)$
\end{tcolorbox}
\caption{Exact widest-path dynamic program on the directed exponent box.}
\label{alg:dp}
\end{figure}

Storing, together with $F(r)$, a predecessor that attains the maximum reconstructs an optimal path.  All comparisons in this recursion are integer comparisons.  The normalised capacity is formed only after the exact bottleneck has been obtained.

\section{Auxiliary repair-floor and frontier certificates}
\label{sec:repair-frontier}

The packing functional fixes the number of insertions and maximises their possible divisor gain.  There is a useful dual question: how cheaply, in multiplicative size, can a prescribed divisor deficit be repaired?  This formulation removes $L_+$ and compares the hypothetical target directly with an alternative record-breaking packet.

\subsection{Repair floors}

\begin{definition}[Restricted repair floor]
\label{def:repair-floor}
Let $H<H'$ be consecutive highly composite numbers and set $g=\gcd(H,H')$.  For a set of primes $\mathcal P$, define
\begin{equation}
 \boxed{
 \mathfrak L_{H,g}(\mathcal P):=
 \min\left\{Q\ge1:
 \supp(Q)\subseteq\mathcal P,\ d(gQ)>d(H)\right\}.}
 \label{eq:repair-floor}
\end{equation}
If no such $Q$ exists, set the value to $+\infty$.  For a prime threshold $t$, write
\begin{equation}
 \mathfrak L_{<t}(H;g):=
 \mathfrak L_{H,g}(\{q\text{ prime}:q<t\}).
 \label{eq:left-repair-floor}
\end{equation}
\end{definition}

\begin{theorem}[Competitive-repair crossing principle]
\label{thm:competitive-repair}
Write $H'=gC$.  Fix a prime threshold $t$ and suppose that, under a hypothetical no-crossing transition, the target-only factor $C$ would be supported on primes $q<t$.  If there exists an integer $D$ such that
\begin{equation}
 d(gD)>d(H)
 \qquad\text{and}\qquad
 D<\mathfrak L_{<t}(H;g),
 \label{eq:competitive-repair-condition}
\end{equation}
then the no-crossing hypothesis is impossible.  Hence the actual target contains an insertion at some prime $q\ge t$.
\end{theorem}

\begin{proof}
Under the no-crossing hypothesis, $C$ is admissible in the minimisation defining $\mathfrak L_{<t}(H;g)$, because $d(gC)=d(H')>d(H)$.  Therefore
\[
 C\ge\mathfrak L_{<t}(H;g)>D.
\]
It follows that $gD<H'=gC$, while $d(gD)>d(H)$.  If $gD\le H$, this contradicts the record property of $H$; if $H<gD<H'$, it contradicts the consecutivity of $H'$.  Thus the target cannot be supported entirely below $t$.
\end{proof}

The competitor $D$ need not be a target state or even a state on a geodesic.  It serves as an arithmetic witness that a left-only repair would be too expensive to define the next record.  This freedom is particularly effective when $D$ is built at an exponent frontier.

\subsection{Prime-exponent frontiers}

\begin{definition}[Prime-exponent frontier]
\label{def:frontier}
For a highly composite number $H=\prod_p p^{a_p}$ and an integer $j\ge1$, define
\begin{equation}
 q_j(H):=\min\{q\text{ prime}:a_q<j\}.
 \label{eq:frontier-definition}
\end{equation}
Every prime below $q_j(H)$ has exponent at least $j$.
\end{definition}

Consider a two-deletion support-loss pattern: the largest support prime $P$ is removed, and a layer $t^A\to t^{A-1}$ is deleted with $A\ge2$.

\begin{lemma}[No-crossing plateau endpoint]
\label{lem:plateau-endpoint}
If no target insertion occurs at a prime $q\ge t$, then
\begin{equation}
 \boxed{q_A(H)=t^+,}
 \label{eq:plateau-endpoint}
\end{equation}
where $t^+$ is the prime immediately following $t$.
\end{lemma}

\begin{proof}
Source monotonicity gives $a_q\ge A$ for every prime $q\le t$.  Under no crossing, no exponent at a prime $q>t$ increases, and the only decreasing coordinates are $t$ and $P$.  If $t^+<P$, target monotonicity and $b_t=A-1$ give
\[
 a_{t^+}=b_{t^+}\le b_t=A-1.
\]
If $t^+=P$, then $a_{t^+}=1\le A-1$.  Thus the first prime whose source exponent is below $A$ is exactly $t^+$.
\end{proof}

The hypothetical obstruction places the relevant frontier directly next to the deletion prime.  By Bertrand's postulate, $t^+<2t$ \citep[Ch.~V]{HardyWright2008}.  Even before one knows whether the frontier layer belongs to the target, it acts as a local half-safe catalyst.

\begin{proposition}[Plateau-frontier catalyst]
\label{prop:frontier-catalyst}
Under the hypotheses of \Cref{lem:plateau-endpoint}, let $q=q_A(H)=t^+$.  The three-step walk
\begin{equation}
 H\longrightarrow\frac HP
 \longrightarrow\frac{Hq}{P}
 \longrightarrow\frac{Hq}{Pt}
 \label{eq:frontier-catalyst-walk}
\end{equation}
never exceeds $H$ and never has fewer than $d(H)/2$ divisors.
\end{proposition}

\begin{proof}
Since $q\le P$, all displayed integers are at most $H$.  The first deletion gives divisor ratio $1/2$.  At the frontier $q_A$, the exponent of $q$ is at most $A-1$, so insertion at $q$ has multiplier at least $(A+1)/A$.  The subsequent deletion at $t$ has multiplier $A/(A+1)$.  The final normalised divisor count is therefore at least
\[
 \frac12\frac{A+1}{A}\frac{A}{A+1}=\frac12,
\]
and the intermediate level is larger.
\end{proof}

The walk may use a layer absent from $H'$, so it is not by itself a geodesic certificate.  Its role is to identify a cheap entropy-restoring packet.  The next theorem turns such packets into lower bounds on the location of a frontier.

\begin{theorem}[Frontier-pinning inequality]
\label{thm:frontier-pinning}
Let $H$ be highly composite and fix $j\ge1$.  Choose distinct primes
\[
 r_1,\ldots,r_k<q_j(H)
\]
and let $U$ be supported away from the $r_i$.  Put
\begin{equation}
 \Gamma_H(U):=\frac{d(HU)}{d(H)}.
\end{equation}
If
\begin{equation}
 \Gamma_H(U)\ge\left(\frac{j+1}{j}\right)^k,
 \label{eq:frontier-gain-condition}
\end{equation}
then
\begin{equation}
 \boxed{U\ge r_1\cdots r_k.}
 \label{eq:frontier-pinning-conclusion}
\end{equation}
\end{theorem}

\begin{proof}
Every $r_i<q_j(H)$ has exponent at least $j$.  Deleting one layer at $r_i$ multiplies the divisor count by
\[
 \frac{a_{r_i}}{a_{r_i}+1}\ge\frac{j}{j+1}.
\]
If $U<r_1\cdots r_k$, then
\[
 M=H\frac{U}{r_1\cdots r_k}<H.
\]
The supports are disjoint, so
\[
 \frac{d(M)}{d(H)}
 =\Gamma_H(U)\prod_{i=1}^k\frac{a_{r_i}}{a_{r_i}+1}
 \ge1.
\]
This contradicts strict recordhood.
\end{proof}

For $j=2$ the theorem says that a packet compensating two second-layer deletions cannot be too cheap unless the first exponent-one prime is correspondingly small.  The following concrete bounds are useful in the finite two-deletion analysis.

\begin{corollary}[Four explicit second-layer frontier bounds]
\label{cor:q2-bounds}
Let $q_2=q_2(H)$.  Then
\begin{alignat}{2}
 (a_2,a_3,a_5)=(6,3,3)&\quad\Longrightarrow\quad q_2\le13,\qquad&
 (a_2,a_3,a_5)=(8,4,3)&\quad\Longrightarrow\quad q_2\le23,\nonumber\\
 (a_2,a_3,a_5)=(7,5,3)&\quad\Longrightarrow\quad q_2\le23,&
 (a_2,a_3,a_5)=(8,5,3)&\quad\Longrightarrow\quad q_2\le29.
 \label{eq:q2-bounds}
\end{alignat}
\end{corollary}

\begin{proof}
Apply \Cref{thm:frontier-pinning} with $j=k=2$.  For $(6,3,3)$, the packet $U=2^2 3^3=108$ has divisor multiplier $9/4$.  If $q_2\ge17$, then $11,13<q_2$ but $11\cdot13=143>108$, a contradiction.

For $(8,4,3)$, the packet $U=2^3 3^2 5=360$ has multiplier $7/3>9/4$; if $q_2\ge29$, then $19\cdot23=437>360$.  For $(7,5,3)$ the same packet has multiplier $55/24>9/4$, giving the same bound.  Finally, for $(8,5,3)$ the packet $U=2^5 3\,5=480$ has multiplier $245/108>9/4$; $q_2\ge31$ would give $23\cdot29=667>480$.
\end{proof}

\section{A generic scheduling counterexample}
\label{app:generic-scheduling-counterexample}

A still more abstract warning is useful.  Total resource and total entropy alone do not force a one-layer barrier.

\begin{proposition}[Generic resource balancing does not imply a half barrier]
\label{prop:generic-scheduling-counterexample}
Consider a resource ceiling $15/2$.  The source has two removable items, each of weight $3$ and value $1$, while the target has three insertable items, each of weight $5/2$ and value $2/3$.  Source and target both have total value $2$, and the target has the larger total weight.  Nevertheless every feasible schedule passes through value $2/3$, which is more than one maximal atomic loss below the starting value.
\end{proposition}

\begin{proof}
Initially no target item fits, because $6+5/2>15/2$, so one source item must be removed.  One target item can then be inserted, but two cannot coexist with the remaining source item: $3+2(5/2)=8>15/2$.  The second source item must therefore be removed while only one target item is present, leaving value $2/3$.
\end{proof}

The arithmetic half barrier, if universal, must consequently use more than a generic knapsack principle.  The special spectrum \eqref{eq:layer-data}, stack precedence, and the record-box gap are genuine arithmetic input.

\section{Detailed certificate audit for the two-deletion stratum}
\label{app:l2-certificate-audit}

The exact reduction of \Cref{thm:l2-reduction} is evaluated independently of the full path recursion.  The results are summarised in \Cref{tab:l2-certificates}.

\begin{table}[t]
\centering
\caption{Certificate accounting for the $301$ transitions with $L_-=2$.}
\label{tab:l2-certificates}
\begin{tabularx}{0.92\textwidth}{@{}Xr@{}}
\toprule
Class or verification & Count\\
\midrule
All two-deletion transitions & $301$\\
Certified directly by the gcd route & $236$\\
Gcd-hard transitions & $65$\\
\quad support-loss patterns & $61$\\
\quad two level-$2$ deletions & $4$\\
Exact-reduction mismatches with full DP & $0$\\
Prime-crossing certificates among hard cases & $65$\\
Source-profile packing certificates among hard cases & $65$\\
Automatic threshold-$2$ support cases & $42$\\
Nontrivial frontier-competition wins & $19$\\
Frontier-competition failures & $0$\\
\bottomrule
\end{tabularx}
\end{table}

The hard deletion-level pairs are
\begin{equation}
 (1,3)^{10},\ (1,5)^4,\ (1,6)^5,\ (1,9)^{22},\
 (1,10)^{17},\ (1,11)^3,\ (2,2)^4,
 \label{eq:hard-pattern-counts}
\end{equation}
where superscripts denote multiplicities.  In all $65$ hard cases the exact source-packing score satisfies
\begin{equation}
 \delta\,G_H(t,m)\le\frac78<1.
 \label{eq:finite-packing-max}
\end{equation}
Thus \Cref{thm:source-packing} forces a target insertion at or beyond the relevant deletion prime, and \Cref{thm:prime-crossing} gives the half-capacity certificate.  For the $19$ support-loss cases with threshold $t>2$, exact repair-floor competition also succeeds; the largest ratio of frontier-assisted repair cost to left-only repair cost is $11/18$.

\section*{Acknowledgements}
\paragraph{AI disclosure.}
OpenAI's ChatGPT was used as a writing and code-assistance tool to reorganise a substantially longer exploratory manuscript, improve exposition, check notation and LaTeX, assist with bibliographic metadata, and help run and review the reproducibility scripts.  Every mathematical statement, proof, computation, citation, and final editorial decision was reviewed by the author, who accepts full responsibility for the contents.

\section*{Data and code availability}
 The data and code supporting this article are archived at Zenodo, version 1.0.0, DOI: \zenododoi.  The archive contains the exact input table, all analysis programs, the figure generator, the computational environment, execution logs, expected outputs, and a SHA-256 manifest.  The arXiv source package contains only the files required to compile the article.

\begingroup
\small
\setlength{\bibsep}{0.35em}
\bibliographystyle{abbrvnat}
\bibliography{references}
\endgroup

\end{document}